\documentclass[11pt]{article}
\usepackage{palatino}
\usepackage{amsmath}
\usepackage{amsthm}
\usepackage{amssymb}
\usepackage{latexsym}
\usepackage{pstricks}
\usepackage{graphicx, pst-plot, pst-node, pst-text, pst-tree}
\usepackage{titlesec}
\usepackage[small,it]{caption}
\usepackage{color}
\definecolor{lightgray}{rgb}{0.8, 0.8, 0.8}
\definecolor{darkgray}{rgb}{0.7, 0.7, 0.7}
\definecolor{darkblue}{rgb}{0, 0, .4}

\usepackage[bookmarks,breaklinks]{hyperref}
\hypersetup{
        colorlinks=true,
        linkcolor=darkblue,
        anchorcolor=darkblue,
        citecolor=darkblue,
        pagecolor=darkblue,
        urlcolor=darkblue,
        pdfpagemode=UseThumbs,
        pdftitle={Almost avoiding permutations},
        pdfsubject={Combinatorics},
        pdfauthor={Robert Brignall, Shalosh B. Ekhad, Rebecca Smith, and Vincent Vatter},
}

\newtheorem{theorem}{Theorem}
\newtheorem{proposition}[theorem]{Proposition}

\newtheorem{conjecture}[theorem]{Conjecture}

\newtheorem{problem}[theorem]{Problem}

\newcommand\blfootnote[1]{%
  \begingroup
  \renewcommand\thefootnote{}\footnote{#1}%
  \addtocounter{footnote}{-1}%
  \endgroup
}

\newcommand{\Av}{\operatorname{Av}}
\newcommand{\C}{\mathcal{C}}
\newcommand{\sh}{\mathrm{sh}\ }

\makeatletter
\newfont{\footsc}{cmcsc10 at 8truept}
\newfont{\footbf}{cmbx10 at 8truept}
\newfont{\footrm}{cmr10 at 10truept}
\renewenvironment{abstract}%
                {
                  \begin{list}{}%
                     {\setlength{\rightmargin}{1in}%
                      \setlength{\leftmargin}{1in}}%
                   \item[]\ignorespaces\begin{small}}%
                 {\end{small}\unskip\end{list}}

\newpagestyle{main}[\small]{
        \headrule
        \sethead[\usepage][][]
        {\sc Almost Avoiding Permutations}{}{\usepage}}

\title{\sc Almost Avoiding Permutations}
\author{
	\begin{tabular}{cc}
        Robert Brignall&Shalosh B. Ekhad\\
		{\small Department of Mathematics}&{\small Department of Mathematics}\\[-3pt]
		{\small University of Bristol}&{\small Rutgers University}\\[-3pt]
		{\small Bristol, UK}&{\small New Brunswick, NJ}\\[20pt]
        Rebecca Smith&Vincent Vatter\thanks{Partially supported by EPSRC grant GR/S53503/01.}\\
        {\small Department of Mathematics}&{\small Department of Mathematics}\\[-3pt]
		{\small SUNY Brockport}&{\small Dartmouth College}\\[-3pt]
		{\small Brockport, NY}&{\small Hanover, NH}\\[20pt]
	\end{tabular}
}

\titleformat{\section}
        {\large\sc}
        {\thesection.}{1em}{}   

\date{June 22, 2009}

\begin{document}
\maketitle

\pagestyle{main}

\begin{abstract}
We investigate the notion of almost avoiding a permutation: $\pi$ {\it almost avoids\/} $\beta$ if one can remove a single entry from $\pi$ to obtain a $\beta$-avoiding permutation.
\end{abstract}

\blfootnote{This paper appeared as \emph{Discrete Math.} 309  (2009), 6626--6631.}

\section{Introduction}\label{almost-intro}

The permutation $\pi$ of length $n$, written in one-line notation as $\pi(1)\pi(2)\cdots\pi(n)$, is said to {\it contain\/} the permutation $\sigma$ if $\pi$ has a subsequence that is order isomorphic to $\sigma$, and each such subsequence is said to be an {\it occurrence\/} of $\sigma$ in $\pi$ or simply a {\/$\sigma$ pattern\/}.  For example, $\pi=491867532$ contains $\sigma= 51342$ because of the subsequence $\pi(2)\pi(3)\pi(5)\pi(6)\pi(9)=91672$.  Permutation containment is easily seen to be a partial order on the set of all (finite) permutations, which we simply denote by $\le$.  If the permutation $\pi$ fails to contain $\sigma$ we say that $\pi$ {\it avoids\/} $\sigma$.

A downset in this permutation containment order is referred to as a {\it permutation class\/}; in other words, if $\C$ is a permutation class, $\pi\in\C$ and $\sigma\le\pi$, then $\sigma\in\C$.  We denote by $\C_n$ the set $\C\cap S_n$ (the permutations of length $n$ in $\C$) and we refer to $\sum_{n\ge 0} |\C_n|x^n$ as the {\it generating function of $\C$\/}.  Given any set of permutations $B$, the set $\Av(B)=\{\pi : \mbox{$\beta\not\le\pi$ for all $\beta\in B$}\}$ forms a permutation class and, conversely, for any permutation class $\C$ there is a unique antichain (set of pairwise incomparable elements) $B$ such that $\C=\Av(B)$; we call this antichain the {\it basis\/} of $\C$, and say that $\C$ is {\it finitely based\/} if $B$ is finite.

One area in which permutation classes arise is the study of sorting machines.  For example, Knuth~\cite{knuth:the-art-of-comp:3} showed the class $\Av(231)$ consists precisely of those permutations that can be sorted by a stack (a last-in first-out linear sorting machine), while Tarjan~\cite{tarjan:sorting-using-n:} observed that the class of permutations that can be sorted by a network consisting of two parallel queues (first-in first-out linear sorting machines) is $\Av(321)$.  A classic result in the field of permutation containment is that the number of permutations in $\Av(231)$ and in $\Av(321)$ of length $n$ are both equal to the $n$th Catalan number.  For bijections between the two sets, see the recent survey by Claesson and Kitaev~\cite{claesson:classification-:}.

Our interest in this paper is with permutations which ``almost lie'' in a given permutation class, a concept which we formalize as follows: given a permutation class $\C$ and natural number $t$, we say that the permutation $\pi$ {\it $t$-almost lies in\/} (or simply {\it almost lies in\/} if $t=1$) $\C$ if one can remove $t$ (or fewer) entries from $\pi$ to obtain a permutation that lies in $\C$.  We denote the set of permutations which $t$-almost lie in $\C$ by $\C^{+t}$; note that $\C^{+t}$ is a permutation class itself.

The notion of almost avoiding permutations has a natural interpretation in terms of sorting machines: if  $\C$ consists of those permutations which can be sorted by the machine $M$ then the permutations in $\C^{+1}$ are those which can be sorted by $M$ in parallel with a {\it one-time use buffer\/}, which we define as a machine which can hold one entry, once in the sorting process.  The classes $\C^{+t}$ for $t\ge 2$ then consist of those permutations which can be sorted by $M$ in parallel with $t$ one-time use buffers.  

Note that almost avoidance classes differ from the classes introduced by Noonan~\cite{noonan:the-number-of-p:}, who studied permutations with at most one copy of $321$; this permutation class, which we denote by $\Av(321^{\le 1})$, is strictly contained in $\Av(321)^{+1}$.

In the following two sections we provide the enumeration of $\Av(231)^{+1}$ and $\Av(321)^{+1}$, i.e., the permutations that can be ``almost stack-sorted'' and ``almost sorted by two parallel queues'', and then end with a conjecture.  By the usual symmetries of permutations, this completes the enumeration of permutations which almost avoid a pattern of length $3$.  Before this, we conclude the introduction with a general result, first proved by the Theory of Computing Research Group at the University of Otago in 2002.

\begin{proposition}[Otago Theory of Computing Research Group]
For any finitely based class $\C$ and positive integer $t$, the class $\C^{+t}$ is finitely based.
\end{proposition}
\begin{proof}
It suffices to prove that $\C^{+1}$ is finitely based if $\C$ is, as then the proposition follows by iteration.  Suppose that the longest basis element of $\C$ has length $m$ and consider a permutation $\tau\notin\C^{+1}$.  Since $\tau\notin\C$, $\tau$ contains a subsequence of length at most $m$ order isomorphic to a basis element of $\C$.  Furthermore, since $\tau\notin\C^{+1}$, every time we remove a single element from this subsequence, we find another occurrence of a basis element of $\C$ (which is also of length at most $m$).  By taking the original subsequence together with these additional occurrences of basis elements we see that $\tau$ contains a permutation of length at most $m(m+1)$ which does not lie in $\C^{+1}$, verifying that the basis of $\C^{+1}$ is finite.
\end{proof}

\section{$\mathbf{Av(321)^{+1}}$}

To enumerate the permutations in $\Av(321)^{+1}$ we use the Robinson-Schensted algorithm.  While a detailed description of this algorithm can be found in Sagan's text~\cite{sagan:the-symmetric-g:}, a few details suffice for our arguments.  First recall that the Robinson-Schensted algorithm associates to each permutation $\pi$ of length $n$ a pair, denoted $(P(\pi),Q(\pi))$, of standard Young tableaux (SYT), each with $n$ cells and of the same shape.  We denote the shape of $P(\pi)$ by $\sh P(\pi)$, so in the case where $\pi$ is of length $n$, $\sh P(\pi)=\sh Q(\pi)$ is a partition, say $\lambda=(\lambda_1,\dots,\lambda_r)$ of $n$ (which we denote by $\lambda\vdash n$).  Schensted~\cite{schensted:longest-increas:} proved that the length of the longest decreasing subsequence $\pi$ is equal to the number of rows of $P(\pi)$, and thus if $\sh P(\pi)=\lambda=(\lambda_1,\dots,\lambda_r)$, then the longest decreasing subsequence of $\pi$ is of length $r$.  Greene~\cite{greene:an-extension-of:} gave a generalization of Schensted's theorem, from which the following proposition routinely follows.

\begin{proposition}\label{greene-321}
If $\sh P(\pi)=(\lambda_1,\dots,\lambda_r)$ then the longest $k\cdots 21$-avoiding subpermutation in $\pi$ has length $\lambda_1+\cdots+\lambda_{k-1}$.
\end{proposition}

The first step in our enumeration is to characterize the shapes of SYT that can arise from a permutation which almost avoids $321$.

\begin{figure}[t]
\begin{center}
\psset{nodesep=5pt,colsep=60pt,rowsep=15pt}
\psset{xunit=0.20in, yunit=0.20in}
\psset{linewidth=0.02in}
$\underbrace{\underbrace{%
\begin{pspicture}(0,0)(9,4)
\psline(0,1)(0,4)
\psline(1,1)(1,4)
\psline(2,2)(2,4)
\psline(3,2)(3,4)
\psline(0,1)(1,1)
\psline(0,2)(3.6,2)
\psline(0,3)(3.8,3)
\psline(0,4)(3.4,4)
\psline(7,2)(7,4)
\psline(8,2)(8,4)
\psline(9,2)(9,4)
\psline(10,3)(10,4)
\psline(11,3)(11,4)
\psline(6.2,2)(9,2)
\psline(6.6,3)(11.7,3)
\psline(6.5,4)(11.4,4)
%
%
%
\psline(15,3)(15,4)
\psline(16,3)(16,4)
\psline(17,3)(17,4)
\psline(14.6,3)(17,3)
\psline(14.3,4)(17,4)
\rput[c](5,3.5){$\cdots$}
\rput[c](13,3.5){$\cdots$}
\end{pspicture}}_{\mbox{\footnotesize{$\ell$ columns}}}%
\begin{pspicture}(0,0)(8,1)
\end{pspicture}%
}_{\mbox{\footnotesize{$k$ columns}}}$
\end{center}
\caption{The shape of a Standard Young Tableaux obtained from the Robinson-Schensted algorithm applied to a permutation in $\Av(321)^{+1}$ that contains at least one $321$ pattern}\label{RowOfThree}
\end{figure}
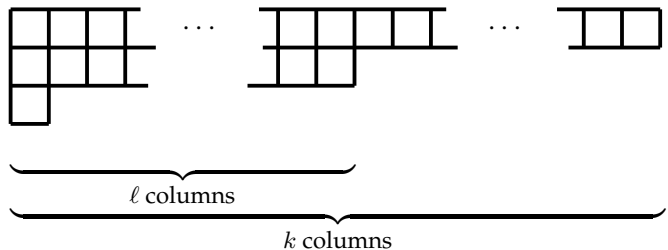

\begin{proposition}\label{almost-321-characterization}
The permutation $\pi$ lies in $\Av(321)^{+1}$ if and only if $\sh P(\pi)$ is of the form $(k)$, $(k,\ell)$, or $(k,\ell,1)$ for some integers $k\ge\ell\ge 1$.
\end{proposition}
\begin{proof}
No permutation in $\Av(321)^{+1}$ can contain a $4321$ pattern as then there would be no entry whose removal gives a $321$-avoiding permutation.  Similarly, no permutation in $\Av(321)^{+1}$ can contain two disjoint occurrences of $321$.  This means that such a permutation cannot contain a $123$-avoiding permutation of length $6$.  Translating into SYT, this means that, for all $\pi\in\Av(321)^{+1}$, $\sh P(\pi)$ has no columns of length $4$ or greater (it avoids $4321$) and at most one column of length $3$ (its longest $123$-avoiding subpermutation has length at most $5$, so this follows from Proposition~\ref{greene-321}).  This implies that $\sh P(\pi)$ is of one of the forms listed in the statement of the proposition.  For the other direction, note that if $\sh P(\pi)=(k)$ or $\sh P(\pi)=(k,\ell)$ then $\pi$ avoids $321$, while if $\sh P(\pi)=(k,\ell,1)$ then $\pi$ contains a subpermutation of length $k+\ell$ which avoids $321$, completing the proof.
\end{proof}

Note that the proof of Proposition~\ref{almost-321-characterization} shows that $\pi\in\Av(321)^{+1}$ if and only if $\pi$ avoids $4321$ and does not contain two disjoint occurrences of $321$.  Thus the basis elements of $\Av(321)^{+1}$ are all of length at most $6$ and can be readily generated by computer.

By Proposition~\ref{almost-321-characterization} and the Robinson-Schensted algorithm we now have that
$$
|\Av_n(321)^{+1}|
=
\sum_{\lambda=(k,\ell)\vdash n} (f^{\lambda})^2
+
\sum_{\lambda=(k,\ell,1)\vdash n} (f^{\lambda})^2,
$$
where $f^\lambda$ denotes the number of SYT of shape $\lambda$.  The first sum is simply the number of $321$-avoiding permutations which, as stated in the introduction, is equal to the $n$th Catalan number, $C_n$.  To evaluate the second sum we use the Hook Length Formula, which states that for $\lambda\vdash n$, $f^\lambda$ is equal to $n!$ divided by the product of the hook lengths of cells in the Ferrers diagram of $\lambda$.  In the case of $\lambda=(k,\ell,1)$, the product of the hook lengths of cells in the top row is
$$
(k+2)k\cdots(k-\ell+2)(k-\ell)\cdots 1
=
\frac{(k+2)k!}{(k-\ell+1)},
$$
the product for the middle row is $(\ell+1)(\ell-1)!$, and the solitary cell in the bottom row has a hook length of $1$.  Thus we have:
\begin{eqnarray*}
|\Av_n(321)^{+1}|
&=&
C_n
+
\sum_{(k,\ell,1)\vdash n}
\left(\frac{n!(k-\ell+1)}{(k+2)k!(\ell+1)(\ell-1)!}\right)^2\\
&=&
C_n
+
\sum_{k=\lfloor n/2\rfloor}^{n-2}
\left(\frac{n!(2k-n+2)}{(k+2)k!(n-k)(n-k-2)!}\right)^2.
\end{eqnarray*}

An empirical calculation in Maple implies that these numbers likely have the generating function
$$
\frac{1-8x+13x^2+24x^3-48x^4-(1-6x+x^2+34x^3-26x^4-4x^5)\sqrt{1-4x}}
{2x^2(1-x)(1-4x)^2}.
$$

\section{$\mathbf{Av(231)^{+1}}$}

Our approach to enumerating the class $\Av(231)^{+1}$ differs significantly from the approach used in the previous section and makes use of the following definition: the entry $\pi(i)$ in $\pi$ is called {\it essential\/} if its removal results in a $231$-avoiding permutation.  For example, the permutation $\pi=1742653$ contains two essential entries, $\pi(5)=4$ and $\pi(7)=3$.  (Also note that if $\pi\in\Av(231)$, then by our definition every entry of $\pi$ is essential.)  As a first step, we make the following observation.

\begin{proposition}\label{prop-essential-elements}
An essential entry in the permutation $\pi\in\Av(231)^{+1}$ participates as the minimum entry in either all or none of the occurrences of $231$ in $\pi$.
\end{proposition}
\begin{proof}
Suppose, to the contrary, that $\pi\in\Av(231)^{+1}$ contains indices $i<j<k$ such $\pi(i)\pi(j)\pi(k)$ is order isomorphic to $231$ and $\pi(k)$ is essential, but that $\pi(k)$ also participates in another $231$ pattern as a non-minimal element.  Label the minimal element of this later pattern $\pi(\ell)$.  Clearly $\pi(i)\pi(j)\pi(\ell)$ is also order isomorphic to $231$, so $\pi-\pi(k)\notin\Av(231)$, a contradiction to the assumption that $\pi(k)$ is essential.
\end{proof}

By Proposition~\ref{prop-essential-elements} we can divide the essential entries of a permutation $\pi\in\Av(231)^{+1}$ into {\it small essential entries\/}, which participate as the minimum entry in all occurrences of $231$, and {\it large essential entries\/}, which participate as the minimum entry in no occurrences of $231$.

In enumerating $\Av(231)^{+1}$, we use two generating functions frequently,
\begin{eqnarray*}
f&=&\mbox{the generating function for $\Av(231)^{+1}$, and}\\
c&=&\mbox{the generating function for the Catalan numbers, so also for $\Av(231)$,}\\
&=&\frac{1-\sqrt{1-4x}}{2x}.
\end{eqnarray*}
Both of these generating functions have constant term $1$.

\begin{proposition}\label{prop-no-greatest}
The generating function for permutations in $\Av(231)^{+1}\setminus\Av(231)$ in which the greatest element is not involved in a copy of $231$ is given by $2(f-c)xc$.
\end{proposition}
\begin{proof}
The plot of a permutation of the specified form can be divided into the greatest entry and two regions, $A$ and $B$, as depicted in Figure~\ref{fig-no-greatest}.  One (but not both) of the two regions $A$ or $B$ must contain a permutation in $\Av(231)^{+1}\setminus\Av(231)$, while the other must contain a $231$-avoiding permutation.  This leads to the generating function specified in the proposition.
\end{proof}

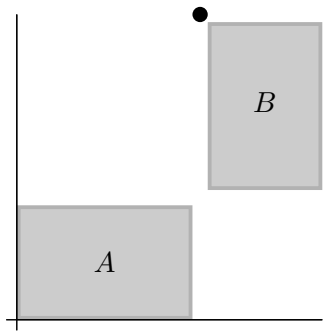
\begin{figure}
\begin{center}
\psset{xunit=0.008in, yunit=0.008in}
\psset{linewidth=0.005in}
\begin{pspicture}(0,-10)(200,210)
\psframe[linecolor=darkgray,fillstyle=solid,fillcolor=lightgray,linewidth=0.02in](0,0)(115,75)
\psframe[linecolor=darkgray,fillstyle=solid,fillcolor=lightgray,linewidth=0.02in](125,85)(200,195)
\pscircle*(120,200){0.04in}
\rput[c](57.5,37.5){$A$}
\rput[c](162.5,142.5){$B$}
\psaxes[dy=1000, Dy=1, dx=1000, Dx=1, tickstyle=bottom, showorigin=false, labels=none](0,0)(200,200)
\end{pspicture}
\end{center}
\caption{A permutation in $\Av(231)^{+1}\setminus\Av(231)$ in which the greatest element is not involved in a copy of $231$.}
\label{fig-no-greatest}
\end{figure}

\begin{proposition}\label{prop-extreme-essential}
The generating function for permutations in $\Av(231)^{+1}$ with an essential greatest (resp., leftmost, rightmost, or least) entry is $x^2c'+xc-c+1$.
\end{proposition}
\begin{proof}
The cases are all similar, so we count permutations with an essential greatest entry.  To construct such a permutation, one must insert a new greatest entry into a $231$-avoider (such permutations have the generating function $x^2c'+xc$ because there are $n+1$ ways to insert a new greatest entry into any $231$-avoider, and the number of $231$-avoiders is the $n$th Catalan number) without creating a $231$-avoider (these have the generating function $c-1$).
\end{proof}

For the following results we need a bit of notation; for a permutation $\pi\in S_n$ and sets $A,B\subseteq[n]$, we write $\pi(A\times B)$ for the permutation which is order isomorphic to the subsequence of $\pi$ which has indices from $A$ and values in $B$.

\begin{proposition}\label{prop-small-essential}
The generating function for permutations in $\Av(231)^{+1}$ with an essential small entry in which the greatest entry participates in a copy of $231$ but is not essential is $x^4c'^2+xc(x^2c'+xc-c+1)$.
\end{proposition}
\begin{proof}
Take $\pi$ of length $n$ specifying the hypotheses, and suppose that $\pi(j)=n$ and that the small essential entry is $\pi(k)$.  As $n$ must be involved in at least one copy of $231$, $\pi([1,j)\times(\pi(k),n))$ must be nonempty; let $\pi(i)$ denote the greatest entry in this region.  By this choice of $i$, $\pi([1,j)\times(\pi(i),n))$ is empty, and because $\pi(k)$ is essential, $\pi((j,n]\times [1,\pi(i)))$ contains only the entry $\pi(k)$.

\begin{figure}
\begin{center}
\psset{xunit=0.008in, yunit=0.008in}
\psset{linewidth=0.005in}
\begin{pspicture}(0,-10)(200,210)
\psframe[linecolor=darkgray,fillstyle=solid,fillcolor=lightgray,linewidth=0.02in](0,0)(75,35)
\psframe[linecolor=darkgray,fillstyle=solid,fillcolor=lightgray,linewidth=0.02in](0,45)(75,80)
\psframe[linecolor=darkgray,fillstyle=solid,fillcolor=lightgray,linewidth=0.02in](85,85)(115,195)
\psframe[linecolor=darkgray,fillstyle=solid,fillcolor=lightgray,linewidth=0.02in](125,85)(200,195)
\pscircle*(20,80){0.04in}
\pscircle*(80,200){0.04in}
\pscircle*(120,40){0.04in}
\rput[c](37.5,17.5){$A$}
\rput[c](37.5,62.5){$B$}
\rput[c](102.5,140){$C$}
\rput[c](162.5,140){$D$}
\psaxes[dy=1000, Dy=1, dx=1000, Dx=1, tickstyle=bottom, showorigin=false, labels=none](0,0)(200,200)
\psline(20,-3)(20,3)
\psline(80,-3)(80,3)
\psline(120,-3)(120,3)
\rput[c](20,-9){$i$}
\rput[c](80,-9){$j$}
\rput[c](120,-9){$k$}
\end{pspicture}
\end{center}
\caption{A permutation with a small essential entry in which the greatest entry is involved in at least one copy of $231$ but is not essential.}
\label{fig-small-essential}
\end{figure}
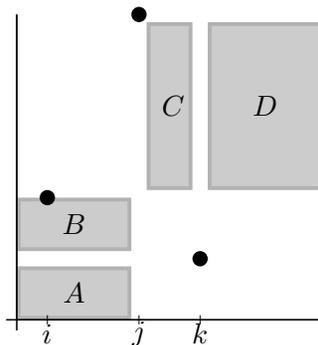

We therefore have three types of entries: the small essential $\pi(k)$, the entries in $\pi([1,j)\times[1,\pi(i)])$, and the entries in $\pi((j,n]\times(\pi(i),n]))$.  We further divide these latter regions into $A,B,C$ and $D$, as indicated in Figure~\ref{fig-small-essential}.

As the entry $n$ is not essential, one of two situations must occur:
\begin{itemize}
\item[(S1)] $\pi$ has an entry in $C$, or
\item[(S2)] $\pi(k)$ forms a copy of $231$ with two entries from $B$.
\end{itemize}
Conversely, any permutation of this form that satisfies (S1) or (S2) is of the desired form.

First we count the permutations satisfying (S1).  In this case the entries in $A\cup B$ form a $231$-avoiding permutation, as do the entries in $C\cup D$; thus both sets of entries are counted by $c$.  The generating function for arrangements of $\pi(k)$ among the entries in $A\cup B$ is therefore $x^2c'$, and this is the same as the generating function for arrangements of $\pi(k)$ among the entries in $C\cup D$.  Multiplying these functions together double-counts the entry $\pi(k)$, but fails to count $\pi(n)$, so the total contribution of the permutations of this type satisfying (S1) is $x^4c'^2$.

Now we need to count the permutations that satisfy (S2) but not (S1).  We know that $\pi$ does not have a entry in $C$ and that the entries of $\pi$ in $D$ avoid $231$.  Furthermore, $\pi(k)$ and the entries in $A\cup B$ form a permutation with an essential rightmost entry, so their generating function is $x^2c'+xc-c+1$ by Proposition~\ref{prop-extreme-essential}.  After taking into account the contribution of $n$, permutations of this type contribute $xc(x^2c'+xc-c+1)$, proving the proposition.
\end{proof}

\begin{proposition}\label{prop-large-essential}
The generating function for permutations in $\Av(231)^{+1}$ without a small essential entry, with a large essential entry that is not the greatest entry, and in which the greatest entry participates in at least one copy of $231$ is given by $x^2c'(x^2c'+xc-c+1)+(x^2c'+x^2c-c+x+1)(c-1)$.
\end{proposition}
\begin{proof}

By arguments analogous to the proof of Proposition~\ref{prop-small-essential} it can be established that these permutations are of the form depicted in Figure~\ref{fig-large-essential} (in this figure the essential large entry is $\pi(i)$).  Note that, as indicated by the figure, $\pi$ must contain a entry in region $C$ as otherwise the greatest element of $\pi$ would not lie in a copy of $231$.  We divide these permutations into two types:
\begin{itemize}
\item[(L1)] there is a copy of $231$ containing $\pi(i)$ and entries in $A\cup B$ or
\item[(L2)] there is no such copy of $231$.
\end{itemize}

\begin{figure}
\begin{center}
\psset{xunit=0.008in, yunit=0.008in}
\psset{linewidth=0.005in}
\begin{pspicture}(0,-10)(200,210)
\psframe[linecolor=darkgray,fillstyle=solid,fillcolor=lightgray,linewidth=0.02in](0,0)(35,80)
\psframe[linecolor=darkgray,fillstyle=solid,fillcolor=lightgray,linewidth=0.02in](45,0)(80,80)
\psframe[linecolor=darkgray,fillstyle=solid,fillcolor=lightgray,linewidth=0.02in](90,85)(200,135)
\psframe[linecolor=darkgray,fillstyle=solid,fillcolor=lightgray,linewidth=0.02in](90,145)(200,195)
\pscircle*(40,140){0.04in}
\pscircle*(85,200){0.04in}
\pscircle*(125,85){0.04in}
\rput[c](17.5,40){$A$}
\rput[c](62.5,40){$B$}
\rput[c](142.5,110){$C$}
\rput[c](142.5,167.5){$D$}
\psaxes[dy=1000, Dy=1, dx=1000, Dx=1, tickstyle=bottom, showorigin=false, labels=none](0,0)(200,200)
\psline(40,-3)(40,3)
\psline(120,-3)(120,3)
\rput[c](40,-9){$i$}
\rput[c](120,-9){$j$}
\end{pspicture}
\end{center}
\caption{A permutation with a large essential entry in which the greatest entry participates in a copy of $231$ but is not essential.}
\label{fig-large-essential}
\end{figure}
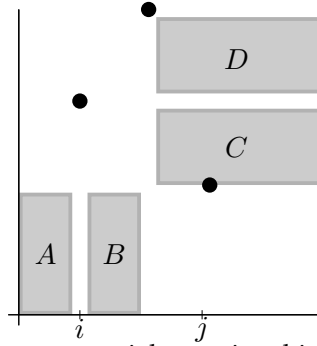

In case (L1), the entries in $A\cup B$ must avoid $231$ because $\pi(i)$ is essential, so they are counted by $c$.  Thus $x^2c'$ is the generating function for the number of arrangements of $\pi(i)$ together with the entries in $A\cup B$.  We then have that $\pi(i)$ is a leftmost essential element of the permutation given by it and the entries in $C\cup D$, so Proposition~\ref{prop-extreme-essential} shows that these entries are counted by $x^2c'+xc-c+1$.  Multiplying these functions double-counts $\pi(i)$ but does not count the greatest entry of $\pi$, so the contribution of the (L1) permutations is $x^2c'(x^2c'+xc-c+1)$.

In case (L2), $\pi(i)$ and the entries in $A\cup B$ form a nonempty $231$-avoiding permutation, and are thus counted by $c-1$.  The entries in $C\cup D$ together with $\pi(i)$ form a permutation with an essential leftmost entry, but we cannot directly apply Proposition~\ref{prop-extreme-essential} because if $\pi$ had a unique entry in $C$ then that entry would be a small essential entry, and we do not wish to count such permutations.  Thus we subtract the generating function for permutations with an essential leftmost entry of value $2$; it is easily seen that this generating function is $x(c-xc-1)$, and so the contribution of permutations in case (L2) is $(x^2c'+x^2c-c+x+1)(c-1)$, completing the proof.
\end{proof}

\begin{theorem}
The generating function for the permutations that almost avoid $231$ is
$$
\frac{1-5x-6x^2+45x^3-24x^4-(1+x-4x^2+x^3)(1-4x)^{3/2}}%
{-2x^2(1-4x)^{3/2}}.
$$
\end{theorem}
\begin{proof}
Letting $f$ denote the generating function for permutations that almost avoid $231$, Propositions~\ref{prop-no-greatest}--\ref{prop-large-essential} show that
\begin{eqnarray*}
\lefteqn{f=c
+
2xc(f-c)
+
x^2c'+xc-c+1
+
x^4c'^2+xc(x^2c'+xc-c+1)
}
\\ & &
+
x^2c'(x^2c'+xc-c+1)+(x^2c'+x^2c-c+x+1)(c-1),
\end{eqnarray*}
from which the desired solution follows.
\end{proof}

\section{Open problems}

Our computations indicate that $|\Av_n(321)^{+1}|<|\Av_n(231)^{+1}|$ for all $n\ge 4$, which begs for a combinatorial explanation:

\begin{problem}
Construct a length-preserving injection from $\Av(321)^{+1}$ to $\Av(231)^{+1}$.
\end{problem}

Finally, we conclude with a conjecture about the exact enumeration problem.

\begin{conjecture}
For all $t$, the generating functions for $\Av(231)^{+t}$ and $\Av(321)^{+t}$ are rational in $x$ and $\sqrt{1-4x}$.
\end{conjecture}

We note that there has been similar work done for sets of permutations with at most $t$ copies of these patterns.  B\'ona~\cite{bona:the-number-of-p:} proved that the generating function for $\Av(231^{\le t})$ is rational in $x$ and $\sqrt{1-4x}$ (see also Mansour and Vainshtein~\cite{mansour:counting-occurr:} and Brignall, Huczynska, and Vatter~\cite{brignall:decomposing-sim:}).  For the pattern $321$, Noonan~\cite{noonan:the-number-of-p:} enumerated $\Av(321^{\le 1})$, while Fulmek~\cite{fulmek:enumeration-of-:} counted $\Av(321^{\le 2})$ and conjectured that the generating function for $\Av(321^{\le t})$ is rational in $x$ and $\sqrt{1-4x}$.

The study of almost avoidance is extended to pairs of permutations of length $3$ by Griffiths, Smith, and Warren~\cite{griffiths:almost-avoiding:}.

\bigskip

{\bf Acknowledgements.}  We would like to thank Mike Atkinson for fruitful discussions and the anonymous referees for their helpful suggestions.

\def\cprime{$'$}

\end{document}